\documentclass[12pt]{amsart}
\usepackage{amsmath}	
\usepackage{amssymb}
\usepackage{caption}
\usepackage{tikz}
\usetikzlibrary{calc, decorations.markings, arrows.meta}

\newtheorem{theorem}{Theorem}[section]
\newtheorem{lemma}[theorem]{Lemma}

\theoremstyle{remark}

\numberwithin{equation}{section}

\newcommand{\bC}{\mathbb{C}}

\newcommand{\bR}{\mathbb{R}}
\newcommand{\cK}{{\mathcal{K}}}
\newcommand{\cL}{{\mathcal{L}}}
\newcommand{\cR}{{\mathcal{R}}}

\newcommand{\dr}{\mathrm{d}\/r}
\newcommand{\dtheta}{\mathrm{d}\/\theta}

\newcommand{\dz}{\mathrm{d}\/z}
\newcommand{\et}{\quad\text{and}\quad}
\newcommand{\tf}{\widetilde{f}}
\newcommand{\tg}{\widetilde{g}}

\renewcommand\Re{\operatorname{Re}}
\renewcommand\Im{\operatorname{Im}}

\begin{document}

\baselineskip=15pt

\title[Paths of steepest descent]
{On the paths of steepest descent for the norm\\ of a one variable
complex polynomial}
\author{Damien Roy}

\thanks{Research partially supported by NSERC}

\begin{abstract}
We consider paths of steepest descent, in the complex plane, for
the norm of a non-constant one variable polynomial $f$.  We show
that such paths, starting from a zero of the logarithmic derivative
of $f$ and ending in a root of $f$, draw a tree in the complex plane,
and we give an upper bound estimate on their lengths.  In some cases,
we obtain a finer estimate that depends only on the set of roots of
$f$, not on their multiplicity, and we wonder if this can be done
in general.  We also extend this question to finite Blaschke products
for the unit disk.
\end{abstract}

\maketitle

%
%

\section{Introduction}
\label{sec:intro}

In our study \cite{Ro2019} of Hermite approximations
to exponentials of algebraic numbers, we were lead to
several results on polynomials of $\bC[z]$ that we could
not find within the vast literature about zeros of
univariate polynomials, including the exposition
\cite{Ma1949} of results prior to 1949, and more recent
papers like \cite{Ru1984}.  The purpose of this note
is to report on these new results with a brief outline of
the proofs, to present a case where the main estimate can
be improved qualitatively, and to ask if such improvement
can be made in general.

To state the results, we fix a monic polynomial
$f(z)\in\bC[z]$ of degree $N\ge 1$, and factor it as a
product
\begin{equation}
\label{eq:f(t)}
 f(z)=(z-\alpha_1)^{n_1}\cdots(z-\alpha_s)^{n_s},
\end{equation}
where $A=\{\alpha_1,\dots,\alpha_s\}$ is the set of
its distinct complex roots and where $n_1,\dots,n_s$ are
positive integers with sum $N$.  We also denote by $\cK$
the convex hull of $A$, and by $R$ the radius of a closed
disk $D$ in $\bC$ containing $\cK$.  To avoid trivialities,
we assume that $s\ge 2$.  In \cite[\S5]{Ro2019}, we prove
the following result.

\begin{theorem}
\label{thm:descent-general}
Any path of steepest descent for $|f|$ linking a point
$\beta$ of $\cK$ to a root $\alpha$ of $f$ is contained
in $\cK$, with length at most $\pi N R$.
\end{theorem}

By a path of steepest descent (resp.\ steepest ascent) for $|f|$,
we mean a continuous piecewise differentiable curve $\gamma\colon I\to\bC$,
defined on an interval $I$ of $\bR$, which is an integral curve
for the gradient of $|f|$, and along which $|f|$ is decreasing
(resp.\ increasing). In the present paper, we show a case where the upper bound
on the length of the path depends only on the set of roots of $f$
and not on their multiplicities.

\begin{theorem}
\label{thm:general-to-boundary}
Suppose that a path of steepest descent for $|f|$ links a point
$\beta$ of $\cK$ to a root $\alpha$ of $f$ on the boundary of $\cK$.
Then, it has length at most $2\pi s R$.
\end{theorem}

Note that the above condition on $\alpha$ is necessary fulfilled
when $s\le 3$.  We wonder if a similar estimate, with an upper bound
depending only on $s$ and $R$, holds in general without the condition that
$\alpha$ lies on the boundary of $\cK$.

The proofs of the above theorems are given in the next section.
In section \ref{sec:tree}, we construct a tree out of the paths
of steepest descent from the zeros of $f'/f$ to the zeros of $f$
and give an application.  We discuss an analog problem for
finite Blaschke products in section \ref{sec:Blaschke}.

\section{Paths of steepest descent}
\label{sec:paths}

Let $c\in\bC\setminus\{0\}$ and let $I$ be a closed
subinterval of $\bR$.  Since $f\colon\bC\to\bC$ is a
ramified covering of Riemann surfaces, there
exists a continuous map $\gamma\colon I\to\bC$ (not unique
in general) such that $f(\gamma(t))=ct$ for each $t\in I$.
More precisely, let $t_0\in I$, let $z_0=\gamma(t_0)$, and
let $\ell$ denote the order of $f(z)-f(z_0)$ at the point
$z_0$.  Then there is a bi-holomorphic map $h\colon U\to V$
from an open neighborhood $U$ of $0$ to an open disk $V$
centered at $0$ of radius $\epsilon>0$, such that
\begin{equation}
 \label{paths:h}
 f(z)-f(z_0)=h(z-z_0)^\ell \quad\text{whenever $z-z_0\in U$.}
\end{equation}
This yields $c(t-t_0)=h(\gamma(t)-\gamma(t_0))^\ell$ for each
$t\in I$ with $|t-t_0|<\epsilon^\ell/|c|$.  If $\ell=1$, then
$\gamma(t)=\gamma(t_0)+h^{-1}(c(t-t_0))$ is an analytic function
of $t-t_0$ for those $t$.  Otherwise, it is represented by a
convergent series in $(t_0-t)^{1/\ell}$ to the left of $t_0$
and by a convergent series in $(t-t_0)^{1/\ell}$ to the right
of $t_0$.  In practice $\gamma$ is obtained by pasting such
local maps (as in \cite[Theorem 4.14]{Fo1981}) and can be
extended to a continuous map $\gamma\colon \bR\to\bC$
satisfying $f(\gamma(t))=ct$ for each $t\in\bR$.

Since $f$ is conformal, except at the zeros of $f'$, since it maps
level curves of $|f|$ to circles centered at $0$, and since the
line $z=ct$ ($t\in\bR$) is perpendicular to these circles, the image
of the above map $\gamma$ is a curve that is perpendicular to the
level curves of $|f|$ at each point $z_0=\gamma(t_0)$ with
$f'(z_0)\neq 0$.  From this observation, we derive the following
statement which shows a connection with the phase plot of $f$, namely
the graph of $f/|f|$ (see \cite{WS2011}).

\begin{lemma}
\label{paths:lemma}
For any $\beta\in\bC$, the continuous functions
$\gamma\colon [0,1]\to\bC$ such that
\begin{equation}
 \label{eq:gamma}
 \gamma(1)=\beta \et f(\gamma(t))=tf(\beta)
 \ \ \text{for each \ $t\in[0,1]$,}
\end{equation}
parameterize (in reverse direction) the curves of steepest descent for
$|f|$ from $\beta$ to a root $\alpha=\gamma(0)$ of $f$.
\end{lemma}

Let $B=\{\beta_1,\dots,\beta_p\}$ denote the set of distinct zeros
of $f'$ not in $A$ (the zeros of the logarithmic derivative $f'/f$).
By a theorem of Gauss-Lucas, these zeros belong to the relative
interior of $\cK$ \cite[Theorem 6.1]{Ma1949}.  In view of \eqref{eq:f(t)},
we find
\begin{equation}
\label{eq:f'(t)}
 f'(z)=N(z-\alpha_1)^{n_1-1}\cdots(z-\alpha_s)^{n_s-1}
 (z-\beta_1)^{m_1}\cdots(z-\beta_p)^{m_p},
\end{equation}
for positive integers $m_1,\dots,m_p$ with sum $s-1$.  For a path
$\gamma\colon[0,1]\to\bC$ as in Lemma \ref{paths:lemma}, we find that
\begin{equation}
\label{eq:gamma'}
 \gamma'(t)=\frac{f(\beta)}{f'(\gamma(t))}
   =\frac{f(\gamma(t))}{tf'(\gamma(t))}
   =\frac{1}{t}\Big(\sum_{j=1}^s\frac{n_j}{\gamma(t)-\alpha_j}\Big)^{-1},
\end{equation}
for each $t\in (0,1)$ with $\gamma(t)\notin B$.  In particular the above
formula holds for each $t\in(0,1)$ such that $\gamma(t)\notin \cK$.
In the latter situation, consider the smallest sector with
vertex at $\gamma(t)$ containing $\cK$.  Then it follows from \eqref{eq:gamma'}
that $\gamma'(t)$ points in the opposite sector, away from $\cK$, as
illustrated in Figure \ref{fig1}, and so from that point, the path
$\gamma$ does not come back to $\cK$.  This means that, if $\gamma(1)$
is a point $\beta$ of $\cK$, the image of $\gamma$ is fully contained
in $\cK$, as asserted in Theorem \ref{thm:descent-general}.

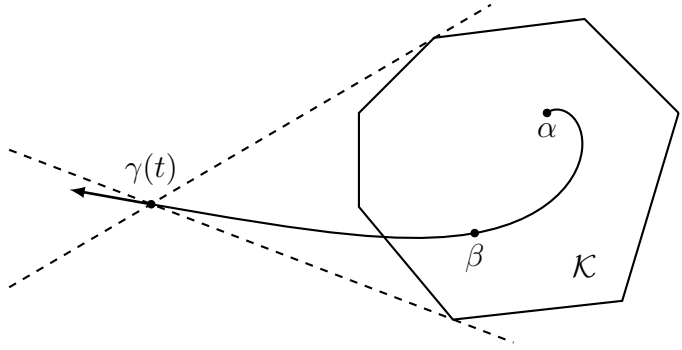
\begin{figure}[h]
 \begin{tikzpicture}[xscale=0.5,yscale=0.5]
       \draw [thick, domain=0:9, samples=40]
             plot ( 40-23*\x: {0.2*\x+\x/(10-\x)});
       \draw[thick] (3.5,0) -- (1,2.5) -- (-3,2) -- (-5,0) -- (-5,-2.5) -- (-2.5,-5.5) -- (2,-5) -- (3.5,0);
       \node[draw,circle,inner sep=1pt,fill] at (0,0) {};
       \draw (0,0) node[below]{$\alpha$};
       \node[draw,circle,inner sep=1pt,fill] at (-121:3.73) {};
       \draw (-121:3.73) node[below]{$\beta$};
       \draw (1,-3.5) node[below]{$\cK$};
       \node[draw,circle,inner sep=1pt,fill] at (-167:10.8) {};
       \draw (-10.52,-2.2) node[above]{$\gamma(t)$};
       \draw[thick, dashed] (-14.3,-4.64)--(-1.49,2.88);
       \draw[thick, dashed] (-14.3,-0.98)--(-0.89,-6.11);
       \draw[-latex,line width=1pt] (-10.52,-2.43)--(-12.7,-2.04);
       \end{tikzpicture}
\caption{A path of steepest ascent that leaves $\cK$ does not come back.}
\label{fig1}
\end{figure}

To estimate the length $\cL(\gamma)$ of a path $\gamma\colon[0,1]\to\bC$
satisfying \eqref{eq:gamma} with $\beta=\gamma(1)\in\cK\setminus A$, we proceed
as in \cite{Ro2019}.  For each angle $\theta\in [0,2\pi]$ and
each $r\in\bR$, we form the line
\begin{equation}
 \label{eq:Drtheta}
 D_{r,\theta}=\{(r+iu)e^{i\theta}\,;\,u\in\bR\}
\end{equation}
and denote by $N(r,\theta)$ the number of points of intersection
of the image of $\gamma$ with this line.  Since $\gamma$ is continuous
and piecewise differentiable, the Cauchy-Crofton formula gives
\begin{equation}
 \label{eq:CC}
 \cL(\gamma)=\frac{1}{4}\int_0^{2\pi} \dtheta \int_{-\infty}^\infty N(r,\theta) \dr
\end{equation}
(see the elegant proof of \cite{AD1997}).  We note that $f(z)/f(\beta)=t\in\bR$ for
each point $z=\gamma(t)$ with $t\in [0,1]$, so $N(r,\theta)$ is at most equal to
the number of roots $u\in\bR$ of the polynomial
\begin{equation}
 \label{eq:grtheta}
 g_{r,\theta}(u)
 = \mathrm{Im}\left(\frac{f\big((r+iu)e^{i\theta}\big)}{f(\beta)}\right)
 \in \bR[u],
\end{equation}
where $\mathrm{Im}(z)$ stands for the imaginary part of $z$.
Since the coefficient of $u^N$ in this polynomial depends only on $\theta$
and does not vanish except for at most $N$ values of $\theta$ in $[0,2\pi)$,
it follows that, aside from these, we have $N(r,\theta)\le N$ for
all $r\in\bR$.  Since the image of $\gamma$ is contained in the disk $D$ with
radius $R$, we also have $N(r,\theta)=0$ outside of an interval of length $2R$
for $r$.  We deduce that $\Big|\int_\bR N(r,\theta) \dr\big| \le 2NR$ for all
but a finite number of $\theta\in[0,2\pi]$, and so $\cL(\gamma)\le \pi N R$.

\begin{proof}[Proof of Theorem \ref{thm:general-to-boundary}]
Suppose now that $\alpha:=\gamma(0)$ lies on the boundary of $\cK$.
Our goal is to show that $\cL(\gamma)\le 2\pi sR$.  To this end, we may
assume that $\cK$ has non-empty interior because otherwise $\cK$ is a line
segment of length at most $2R$ and we are done.  Under that hypothesis,
the point $\beta$ belongs to that interior as well as $\gamma(t)$ for each
$t\in(0,1]$.

We first extend $\gamma$ to a continuous function $\gamma\colon\bR\to\bC$
satisfying $f(\gamma(t))=tf(\beta)$ for each $t\in\bR$.  Then there exists
a unique point $c>1$ for which $\gamma(c)$ lies on the boundary of $\cK$: for
$0<t<c$, the point $\gamma(t)$ belongs to the interior of $\cK$ and, for $t>c$,
it is outside of $\cK$ (the proof uses \eqref{eq:gamma'} as above). Then, we get
a simple closed curve $\Gamma$ by following $\gamma$ on $[0,c]$ and coming back
to $\alpha=\gamma(0)$ along the polygonal boundary of $\cK$. By a theorem of
Jordan, $\Gamma$ divides $\bC$ into two connected components $\cR_1$ and $\cR_2$
having $\Gamma$ as their common boundary.

Now, consider an angle $\theta\in[0,2\pi]$.  We claim that there are finitely
many real numbers $r$ for which $D_{r,\theta}$ meets $A\cup B$ or is tangent
to $\Gamma$ at a point $\gamma(t)\in\cK\setminus B$ with $t\in[0,c]$.  This is
because, for each $t_0\in\bR$, there exists $\epsilon>0$ such that the restriction
of $\gamma(t)$ to $(t_0-\epsilon,t_0]$ and $[t_0,t_0+\epsilon)$ is given by convergent
power series in $|t-t_0|^{1/\ell}$ where $\ell-1$ is the order of $f'$ at $\gamma(t_0)$.
Then, the function
\begin{equation}
 \label{eq:h}
 h(t)=\Re(\gamma'(t)e^{-i\theta}),
\end{equation}
where $\Re(z)$ stands for the real part of $z$, is given by convergent
power series in $|t-t_0|^{1/\ell}$ on $(t_0-\epsilon,t_0)$ and on $(t_0,t_0+\epsilon)$.
On such an interval $I$, either $h$ is identically zero or it has finitely many zeros.
Thus, $\Re(\gamma(t)e^{-i\theta})$ takes finitely many values on the set of zeros of
$h$ on $I$ and these are precisely the values $r\in\bR$ for which $D_{r,\theta}$ is
tangent to the open arc $\gamma(I)$.  The claim follows since $[0,c]$ is compact and so
it can be covered by finitely many intervals $(t_0-\epsilon,t_0+\epsilon)$ with those
properties.

Suppose that $(ie^{i\theta})^N\notin\bR$, so that the polynomial $g_{r,\theta}$
given by \eqref{eq:grtheta} has degree $N$ for each $r\in\bR$.  Then $D_{r,\theta}$
meets $\gamma([0,c])$ in at most $N$ points
\[
 z_j=\gamma(t_j)=(r+iu_j)e^{i\theta}  \quad (1\le j\le n:=N(r,\theta)),
\]
which we order so that $u_1<u_2<\cdots<u_n$.  Moreover, aside from finitely many values
of $r$, the line $D_{r,\theta}$ avoids $A\cup B$ and meets $\Gamma$ transversally
at those points: the function $h$ given by \eqref{eq:h} is defined and non-zero at
$t_1,\dots,t_n$.  For each $j$ with $1\le j<n$, the open line segment $(z_j,z_{j+1})$
on $D_{r,\theta}$ does not meet $\Gamma$ and so is contained in one of the connected
regions $\cR_1$ or $\cR_2$.  Geometrically, this means that $\gamma$ crosses $D_{r,\theta}$
in opposite directions at $z_j$ and at $z_{j+1}$ so to have this line segment on
the same side at both points, as illustrated on Figure \ref{fig2}.  Thus $h(t_j)$
and $h(t_{j+1})$ have opposite signs.  Using \eqref{eq:gamma'}, we find that
$t\gamma'(t)=P(\gamma(t))/Q(\gamma(t))$ where
\[
 P(z)=(z-\alpha_1)\cdots(z-\alpha_s)
 \et
 Q(z)=N(z-\beta_1)^{m_1}\cdots(z-\beta_p)^{m_p}
\]
have degree $s$ and $s-1$ respectively.  Since each $t_j$ is positive, it
follows that, for $j=1,\dots,n$, the number $h(t_j)$ has the same sign
as $\tg_{r,\theta}(u_j)$ where
\[
 \tg_{r,\theta}(u)
  = \Re\left(e^{-i\theta}P\big((r+iu)e^{i\theta}\big)
         \overline{Q}\big((r-iu)e^{-i\theta}\big)\right).
\]
Thus the polynomial $\tg_{r,\theta}$ alternates sign at $u_1,\dots,u_n$.
Since it has degree at most $2s-1$, we conclude that $N(r,\theta)=n\le 2s$.
This implies that $\int_\bR N(r,\theta)\dr\le 4sR$ and so
the Cauchy-Crofton formula \eqref{eq:CC} yields an upper bound of
$2\pi sR$ for the length of $\gamma([0,c])$.
\end{proof}

\begin{figure}[h]
 \begin{tikzpicture}[xscale=0.7,yscale=0.45]
       \node[draw,circle,inner sep=1pt,fill] at (5.38,3.86) {};
       \draw (5.38,3.86) node[above right]{$\alpha$};
       \draw[thick] (-6,-6) -- (-7,-2) -- (-5,4) -- (0,6) -- (5.38,3.86) -- (6,-2) -- (1,-6) -- (-6,-6);
       \draw [-{Latex[length=3mm,width=2mm]},thick] (-3,-6) -- (-2.9,-6);
       \draw [-{Latex[length=3mm,width=2mm]},thick] (3.5,-4) -- (3.6,-3.92);
       \draw (3.5,-4) node[below right] {$\Gamma$};
       \draw [-{Latex[length=3mm,width=2mm]},thick, domain=0.045:0.1, samples=40]
             plot ( {240+540*(1-\x)} : {5*(1-\x)+5*(\x-0.2)^2/(0.02+\x)} );
       \draw [thick, domain=0.1:0.2, samples=40]
             plot ( {240+540*(1-\x)} : {5*(1-\x)+5*(\x-0.2)^2/(0.02+\x)} );
       \draw [-{Latex[length=3mm,width=2mm]},thick, domain=0.2:0.42, samples=40]
             plot ( {240+540*(1-\x)} : {5*(1-\x)} );
       \draw [thick, domain=0.42:1, samples=40]
             plot ( {240+540*(1-\x)} : {5*(1-\x)} );
       \node[draw,circle,inner sep=1pt,fill] at (240:0) {};
       \draw (240:0) node[above right]{$\beta$};
       \draw [-{Latex[length=3mm,width=2mm]},thick, domain=0:0.3, samples=40]
             plot ( 60+540*\x: {5*\x});
       \draw [-{Latex[length=3mm,width=2mm]},thick, domain=0.3:0.59, samples=40]
             plot ( 60+540*\x: {5*\x});
       \draw [thick, domain=0.59:0.8, samples=40]
             plot ( 60+540*\x: {5*\x});
       \draw [-{Latex[length=3mm,width=2mm]},thick, domain=0.8:0.925, samples=40]
             plot ( 60+540*\x: {5*\x+10*(\x-0.8)^2/(1.04-\x)});
       \draw [thick, domain=0.925:0.96, samples=40]
             plot ( 60+540*\x: {5*\x+10*(\x-0.8)^2/(1.04-\x)});
       \node[draw,circle,inner sep=1pt,fill] at (-6.27,-4.97) {};
       \draw (-6.27,-4.97) node[left]{$\gamma(c)$};
       \draw [thick] (-8,-3.5) -- (8,0.5);
       \draw (-3,-4) node[left] {$\cR_1$};
       \draw (4,3) node[left] {$\cR_2$};
       \draw (8,0.5) node[below]{$D_{r,\theta}$};
       \end{tikzpicture}
\caption{The curve $\Gamma$ in the proof of Theorem \ref{thm:general-to-boundary}}
\label{fig2}
\end{figure}
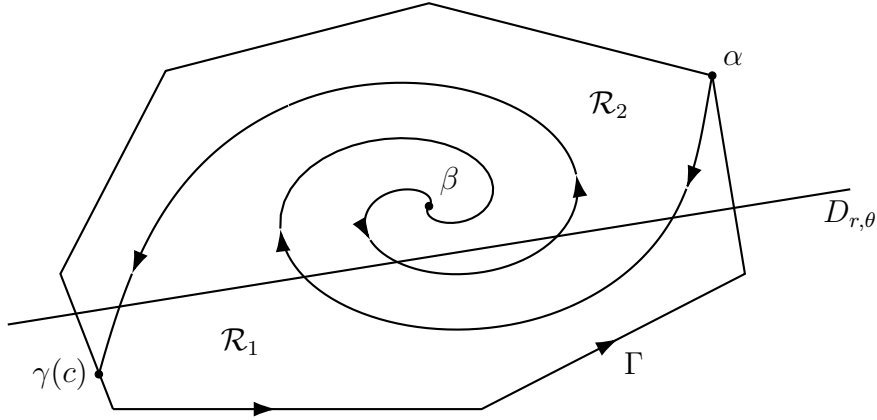

In general, any continuous map $\gamma\colon[0,1]\to\cK$ satisfying
$f(\gamma(t))=tf(\beta)$ for each $t\in[0,1]$ and some $\beta\in\cK$
extends to a continuous map $\gamma\colon[a,c]\to\cK$ satisfying
the same condition on a maximal interval $[a,c]$ containing $[0,1]$,
and we obtain a simple closed curve $\Gamma$ dividing $\bC$ in two
connected components by following $\gamma$ on $[a,c]$ and coming back to
$\gamma(a)$ along the boundary of $\cK$.  Then most of the argument
goes through, except that, if $\gamma(0)$ is not on the boundary of
$\cK$, then $a$ is negative and, for points $t_j$ with
$a<t_j<0$ (if any), the sign of $h(t_j)$ is opposite to that of
$\tg_{r,\theta}(u_j)$.  It may even happen that all $\tg_{r,\theta}(u_j)$
have the same sign if $\gamma$ winds several times around $\gamma(0)$
on $[a,0]$ and unwinds an equal number of times on $[0,c]$ so that $t_j$ and
$t_{j+1}$ have opposite signs for $j=1,\dots,n-1$.

\section{A tree of paths}
\label{sec:tree}

We construct a graph on the set $A\cup B$ in the following way.  For each
$j=1,\dots,p$, the difference $f(z)-f(\beta_j)$ has a zero of multiplicity
$\ell=m_j+1$ at $\beta_j$ and so it is locally the $\ell$-th power of
a diffeomorphism as in \eqref{paths:h} (with $z_0=\beta_j$).  It follows
that, in a sufficiently small neighborhood of $\beta_j$, there are exactly
$\ell$ paths of steepest descent issued from $\beta_j$ and their tangents
make equal angles of $2\pi/\ell$ at that point.  Each of these paths can be continued
uniquely until it reaches another point of $A\cup B$.  Doing this for
each $j=1,\dots,p$, we obtain a total of $\sum_{j=1}^p(m_j+1)=s+p-1$ paths
from a point of $B$ to a point in $A\cup B$.  These paths intersect only
on their end-points because there is only one path of steepest descent
through each point not in $A\cup B$.  So they make the edges of
a graph with $A\cup B$ as its set of vertices.

We claim that this graph has no cycle.  Indeed, if there were a cycle,
we would obtain a simple closed curve $\Gamma$ by composing a number
of these paths.  Consider the bounded connected region $\cR$ delimited by
$\Gamma$.  By the maximum modulus principle, the maximum of $|f|$ on $\cR$
is achieved at a point of $\Gamma$ and so at some $\beta_j$ on $\Gamma$.
This is impossible because, in a neighborhood of $\beta_j$, the curve
$\Gamma$ consists of two paths of steepest descent issued from $\beta_j$.
However, in each of the two angles formed by these paths at $\beta_j$,
there is at least one path of steepest ascent for $|f|$ issued from
$\beta_j$ (locally, there are exactly $m_j+1$ such paths and their tangents
at $\beta_j$ are bisectors of the angles formed by the $m_j+1$
tangents to the paths of steepest descent from $\beta_j$).  This is
impossible since one of these paths would enter in $\cR$,
and $|f|$ increases along such a path.

Since the graph has no cycle and its number of edges is $s+p-1$, one less
than the cardinality of its set of vertices $A\cup B$, we conclude that
it is a tree (a connected graph with no cycle).

In \cite[\S8]{Ro2019}, we use the above to estimate some integrals. To present
this application, fix some $\beta_j\in B$ and consider two paths of steepest
descent for $|f|$ from $\beta_j$ to roots of $f$.  Since the graph constructed
above has no cycle, these paths have distinct end-points, say $\alpha_i$ and
$\alpha_k$ with $i\neq k$.  By composing these paths, we obtain a continuous map
$\gamma\colon[0,1]\to\cK$ with $\gamma(0)=\alpha_i$, $\gamma(1/2)=\beta_j$
and $\gamma(1)=\alpha_k$ such that $|f(\gamma(t))|$ is maximal equal to
$|f(\beta_k)|$ at $t=1/2$.  According to Theorem \ref{thm:descent-general},
the length of $\gamma$ is at most $2\pi NR$, and so we obtain
\[
 \left|\int_{\alpha_i}^{\alpha_k} f(z)e^{-z}\dz\right|
   \le \cL(\gamma)\max_\gamma|f|\max_\cK |e^{-z}|
   \le 2\pi NRe^R |f(\beta_j)|
\]
by integrating along $\gamma$, and assuming the disk $D$ centered at $0$.  
In view of Theorem \ref{thm:general-to-boundary},
if $\alpha_i$ and $\alpha_k$ lie on the boundary of $\cK$, we may replace
the factor $N$ by $2s$ in this estimate.  We wonder if one can replace
$N$ by a function of $s$ and $R$ in general.

Note that, for $r=|f(\beta_j)|$, the above points $\alpha_i$
and $\alpha_k$ belong to disjoint connected components of the set
$U_r:=\{z\in\bC\,;\,|f(z)|<r\}$ because otherwise we could construct a
simple closed curve passing through $\beta_j$, contained in
$\gamma([0,1])\cup U_r$, and that would violate
the maximum modulus principle for $f$.  Thus, along any
continuous path linking $\alpha_i$ and $\alpha_k$, the maximum of $|f|$ is at
least equal to $|f(\beta_j)|$.  In general, for any $r>0$, the number
of connected components of $U_r$ is one more than the number of zeros
of $f'$ outside of $U_r$, counting multiplicities \cite[\S2]{Wa1935}.

\section{Finite Blaschke products}
\label{sec:Blaschke}

Many of the above observations apply as well to finite Blaschke products mapping
the open unit disk $D=\{z\in\bC\,;\,|z|<1\}$ to itself.  These are rational
functions of the form
\[
 f(z)=c\prod_{j=1}^s\Big(\frac{z-\alpha_j}{1-\overline{\alpha}_jz}\Big)^{n_j}
\]
with $|c|=1$, where $A=\{\alpha_1,\dots,\alpha_s\}$ is the set of distinct
zeros of $f$ in $D$, and where $n_j\ge 1$ denotes the multiplicity of
$\alpha_j$ for $j=1,\dots,s$.  Set $N=n_1+\dots+n_s$ and suppose for simplicity
that $f$ is normalized so that $f(0)=0$ and $f(1)=1$.  Then, its derivative $f'$
has exactly $N-1$ zeros in $D$ counting multiplicities, and these in turn completely
determine $f$ (see \cite{SW2019} for details and references).  For the set
$B$ of zeros $\beta_1,\dots,\beta_p$ of $f'$ in $D\setminus A$ and
their respective multiplicities $m_1,\dots,m_p$, this means in particular that
$m_1+\dots+m_p=s-1$.

Paths of steepest descent for $|f|$ starting on a point of $D$ remain in $D$
and may be continued until they reach a zero of $f$ in $A$, since $|f|$ decreases
along such a path.  In particular, we may form all paths of steepest descent
from an element of $B$ down to a first new element of $A\cup B$.  Following
the argument of section \ref{sec:tree}, this yields $p+s-1$ curves which draw
the edges of a tree on $A\cup B$.

Finally, fix any path of steepest descent for $|f|$ from a point $\beta\in
D\setminus A$ to a zero of $f$ in $A$.  We claim that its length is at most $2\pi N$.
Since $D$ has radius $1$, the Cauchy-Crofton formula reduces this to showing
that, except for finitely many angles $\theta\in [0, 2\pi]$, the lines $D_{r,\theta}$
with $r\in\bR$ given by \eqref{eq:Drtheta} meet the path in at most $2N$ points.
Note that, as $f(0)=0$, we may assume that $\alpha_1=0$ and so,
for each $z\in D$, the complex number $f(z)$ is a real multiple of
\[
 \tf(z)=cz^{n_1}\prod_{j=2}^s\big((z-\alpha_j)(1-\alpha_j\overline{z})\big)^{n_j}.
\]
By Lemma \ref{paths:lemma}, the ratio $f(z)/f(\beta)$ is a real number at
each point $z$ along the path, and thus $\Im(\tf(z)/f(\beta))=0$.  So, for given
$\theta\in [0, 2\pi]$ and $r\in\bR$, the number of points of the path on the line
$D_{r,\theta}$ is at most equal to the number of real zeros
of the polynomial $\Im(\tf((r+iu)e^{i\theta})/f(\beta))\in \bR[u]$.  Since this polynomial
has degree at most $2N-n_1$ and since its coefficient of $u^{2N-n_1}$ is
$\Im(c'e^{in_1\theta})$ with a constant $c'\neq 0$, the conclusion follows.

Again, we wonder if there is an upper bound for that length which depends only on $s$.

\subsection*{Acknowledgments}
The author thanks the anonymous referee for his suggestions and
for mentioning the interesting techniques of Ruscheweyh in
\cite{Ru1984}.  He also thanks the journal editor Javad Mashreghi
for the reference \cite{WS2011} and his suggestion to look at
finite Blaschke products.

\vfill

\small
\vbox{
\hbox{Damien \sc Roy}\par
\hbox{D\'epartement de math\'ematiques et de statistique}\par
\hbox{Universit\'e d'Ottawa}\par
\hbox{150 Louis Pasteur}\par
\hbox{Ottawa, Ontario}\par
\hbox{Canada K1N 6N5}
}

\end{document}